\documentclass[11pt]{amsart}

\usepackage{epigamath}

\usepackage[english]{babel}

\numberwithin{equation}{section}

\usepackage{amsmath, amsthm, amssymb,amscd}
\usepackage{float}
\usepackage{graphicx}
\usepackage{xypic}
\usepackage[arrow, matrix, curve]{xy}
\usepackage{tikz-cd}

\newtheorem{thm}{Theorem}[section]
\newtheorem{lem}[thm]{Lemma}
\newtheorem{prop}[thm]{Proposition}

\newcommand{\seq}{\subseteq}

\def\PP{{\mathbf P}}

\DeclareMathOperator{\Pic}{Pic}
\DeclareMathOperator{\Proj}{Proj}
\DeclareMathOperator{\id}{id}
\DeclareMathOperator{\Coker}{Coker}
\DeclareMathOperator{\Sym}{Sym}

\newcommand{\lra}{\longrightarrow}

\newcommand{\longtwoheadrightarrow}{\relbar\joinrel\twoheadrightarrow}
\newcommand{\longhookrightarrow}{\lhook\joinrel\longrightarrow}

\EpigaVolumeYear{9}{2025} \EpigaArticleNr{27} \ReceivedOn{March 6, 2025}
\InFinalFormOn{June 24, 2025}
\AcceptedOn{July 28, 2025}

\title{A proof of generic Green's conjecture in odd genus}
\titlemark{A proof of generic Green's conjecture in odd genus}

\author{Michael Kemeny}
\address{University of Wisconsin-Madison, Department of Mathematics, 480 Lincoln Drive, Madison, WI 53706, USA} 
\email{michael.kemeny@gmail.com}

\authormark{M.~Kemeny}

\AbstractInEnglish{In this note, we give a new proof of Voisin's
  theorem on canonical syzygies for generic curves of odd genus.}

\MSCclass{14H99}

\KeyWords{Curves, syzygies, Green's conjecture}

\acknowledgement{The author was supported by NSF grant DMS-2100782.}

\begin{document}

%%%%%%%%%%%%%%%%%%%%%%%%%%%%%%%
% Title page
%%%%%%%%%%%%%%%%%%%%%%%%%%%%%%%

%\removeabove{}
%\removebetween{}
%\removebelow{}

\maketitle

\begin{prelims}

\DisplayAbstractInEnglish

\bigskip

\DisplayKeyWords

\medskip

\DisplayMSCclass

%\bigskip

%\languagesection{Fran\c{c}ais}

%\bigskip

%\DisplayTitleInFrench

%\medskip

%\DisplayAbstractInFrench

\end{prelims}

%%%%%%%%%%%%%%%%%%%%%
% Table of Contents
%%%%%%%%%%%%%%%%%%%%%

\newpage

\setcounter{tocdepth}{1}

\tableofcontents

%%%%%%%%%%%%%%%%%%%%%
% Content begins here
%%%%%%%%%%%%%%%%%%%%%

\section{Introduction}

Our goal is to give a proof of Voisin's theorem in \cite{voisin-thm} resembling \cite[Sections~1 and~2]{kemeny-voisin}. Green's conjecture on the syzygies of a canonical curve, see \cite{green-koszul},  is one of the central conjectures on the syzygies of projective varieties. It was first proven for a generic curve by Voisin. This proof, using the geometry of $K3$ surfaces intimately, is a landmark result; however, it is somewhat long and complicated. Alternate proofs have been offered recently, including an approach using the tangent developable, see \cite{AFPRW}, which is very different to the approach we take here. Our approach here is rather formal and only uses the geometry of $K3$ surfaces for the first few steps to set up the problem, leading one to hope that it generalizes to new situations.

\subsection*{Acknowledgments}
I thank the referee for a careful reading.

\section{The proof}
The main idea is that while, in the notation below, the local complete intersection scheme $\mathcal{Z}$  is not finite over $\PP$,  $\widetilde{\mathcal{Z}}$ is finite over $\PP_2$. 

Let $X$ denote a $K3$ surface with $\Pic(X)$ generated by a big and nef line bundle $L'$ with $(L')^2=4k+2$ together with a smooth rational curve $\Delta$ with $(L' \cdot \Delta)=0$ for $k\geq 4$; so we need a different $K3$ surface for each integer $k$.\footnote{There is a typo in \cite[Section~3]{kemeny-voisin}, where it is written $(L')^2=2k+2$.}  Set $L:=L'(-\Delta)$.  Then $L$ and $L'$ are base-point-free, with $\mathrm{H}^1(X,L)=\mathrm{H}^1(X,L')=0$; see \cite[Lemma 6]{kemeny-voisin}. 
We use the \emph{kernel bundle} approach, see \cite[Section~3]{ein-lazarsfeld-asymptotic}, and so need to show
$\mathrm{H}^1(\bigwedge^k M_L(L))=0$. We have the exact sequence
$$
0 \lra \bigwedge^{k+1}M_L (L') \lra \bigwedge^{k+1} M_{L'}(L') \lra \bigwedge^k M_L(L) \lra 0,
$$
where $M_L$, $M_{L'}$ are the kernel bundles in the notation of \cite[Section~2.1]{aprodu-nagel}. Duality gives $\mathrm{H}^2 (\bigwedge^{k+1}M_L (L'))\simeq \mathrm{H}^0 (\bigwedge^{k}M_L (-\Delta))\seq \bigwedge^k \mathrm{H}^0(L) \otimes \mathrm{H}^0(\mathcal{O}(-\Delta))=0$. So $\mathrm{H}^1(\bigwedge^k M_L(L))=0$ if the map $\mathrm{H}^1 (\bigwedge^{k+1}M_L (L')) \to  \mathrm{H}^1 (\bigwedge^{k+1}M_{L'} (L'))$,
induced from $M_L \seq M_{L'}$, is surjective.

 Contract $\Delta$ via a map $\mu\colon X \to \hat{X}$, where $\hat{X}$ is a nodal $K3$ surface. We have a line bundle $\hat{L}$ on $\hat{X}$ with $\mu^* \hat{L} \simeq L'$ and the rank two bundle $\hat{E}$ on $\hat{X}$, with $h^0(\hat{X}, \hat{E})=k+3$ and $\det \hat{E}=\hat{L}$, defined as the Lazarsfeld--Mukai bundle induced by a general $g^1_{k+2}$ on $C \in |\hat{L}|$. Set $E$ to be the vector bundle $\mu^*\hat{E}$ on~$X$.  Consider $E(-\Delta)$, a globally generated bundle with $h^0(X, E(-\Delta))=k+1$; see \cite[Section~3]{kemeny-voisin}. Then $\mathrm{H}^0(X,E)\simeq \mathrm{H}^0(\hat{X},\hat{E})$.  Set $\PP:=\PP(\mathrm{H}^0(X,E)):=\Proj(\mathrm{H}^0(X,E)^{\vee})
 \simeq \PP^{k+2}$. 
Let $p\colon X\times \PP \to X$ and $q\colon X\times \PP \to \PP$ denote the projections. Let $\mathcal{Z}\seq X \times \PP$ be the locus 
$\{ (x,s)\mid s(x)=0 \} \seq X \times \PP$, which is a projective bundle over $X$, and thus irreducible. We have $\dim \mathcal{Z}=k+2$, and $q_{|_{\mathcal{Z}}}$ is generically finite but has one-dimensional fibres over $\PP(\mathrm{H}^0(X,E(-\Delta))$. Further, $\mathcal{Z}$ has codimension two in $X \times \PP$, and we have an exact sequence
$$
0 \lra \mathcal{O}_X \boxtimes \mathcal{O}_{\PP}(-2) \overset{\id}\lra E \boxtimes \mathcal{O}_{\PP}(-1) \lra p^*L'\otimes I_{\mathcal{Z}} \lra 0,
$$
which is a Koszul complex, where $\id$ is given by multiplication by $\id$ in
$$
\mathrm{H}^0(E \boxtimes \mathcal{O}_{\PP}(1)) \simeq \mathrm{H}^0(E) \otimes  \mathrm{H}^0(E)^{\vee}.
$$
Set $\PP_2:=\PP(\mathrm{H}^0(\hat{X},\hat{E}))$. Then $\mu_*E=\mu_*\mu^*\hat{E} \simeq \hat{E}$, and we have an isomorphism $i \colon \PP \to \PP_2$.
Let $\hat{\mathcal{Z}} \seq \hat{X} \times \PP_2$ be the codimension two locus $\{ (x,s) \mid s(x)=0 \} \seq \hat{X} \times \PP(\mathrm{H}^0(\hat{X},\hat{E}))$,  which is an lci in a Cohen--Macaulay scheme and hence a Cohen--Macaulay scheme; see \cite[p.\ 112]{matsumura}.

 We have an exact sequence 
$0 \to \mathcal{O}_{\hat{X}} \boxtimes \mathcal{O}_{\PP_2}(-2) \overset{\id}\lra \hat{E} \boxtimes \mathcal{O}_{\PP_2}(-1) \to {\hat{p}}^*{\hat{L}}\otimes I_{\hat{\mathcal{Z}}} \to 0$,
 where $\hat{p} \colon \hat{X} \times \PP_2 \to \hat{X}$ and $\hat{q} \colon \hat{X} \times \PP_2 \to \PP_2$ denote the projections. We let
 $$
 \tau:=\mu \times i \colon X \times \PP \lra \hat{X} \times \PP_2.
 $$
 Let $v$ be the node of $\hat{X}$.

\begin{lem} \label{Z-lemma}
We have $\tau^* I_{\hat{\mathcal{Z}}}\simeq I_{\mathcal{Z}}$, $\tau_* I_{\mathcal{Z}}\simeq I_{\hat{\mathcal{Z}}}$ and $\tau_*(I_{\mathcal{Z}}\otimes p^*\mathcal{O}_X(-\Delta)) \simeq I_{\hat{{\mathcal{Z}}}} \otimes \hat{p}^*I_v$.
\end{lem}

\begin{proof}
We have the two exact sequences
\begin{align*}
0 &\lra \mathcal{O}_X \boxtimes \mathcal{O}_{\PP}(-2) \overset{\id}\lra E \boxtimes \mathcal{O}_{\PP}(-1) \lra p^*L'\otimes I_{\mathcal{Z}} \lra 0 \quad\text{and} \\
0 &\lra \mathcal{O}_{\hat{X}} \boxtimes \mathcal{O}_{\PP_2}(-2) \overset{\id}\lra \hat{E} \boxtimes \mathcal{O}_{\PP_2}(-1) \lra {\hat{p}}^*{\hat{L}}\otimes I_{\hat{\mathcal{Z}}} \lra 0.
\end{align*}
Now $\mu_* \mathcal{O}_{X}\simeq \mathcal{O}_{\hat{X}}$, $R^i \mu_* \mathcal{O}_{X}\simeq \mathcal{O}_{\hat{X}}=0$ for $i>0$, $\mu^*\hat{E} \simeq E$ and $\mu_*E \simeq \hat{E}$. So the first two claims follow by applying $\tau^*$ or $\tau_*$ to the two exact sequences. For the third statement, we have the exact sequence
$$
0 \lra \mathcal{O}_X(-\Delta) \boxtimes \mathcal{O}_{\PP}(-2) \overset{\id}\lra E(-\Delta)  \boxtimes \mathcal{O}_{\PP}(-1) \lra p^*L'(-\Delta) \otimes I_{\mathcal{Z}} \lra 0.
$$
We may view $\tau$ as the blow-up of $\hat{X} \times \PP_2$ in the singularity $v \times \PP_2$. The claim then follows by applying $\tau_*$.
\end{proof}

The following lemma is important for us.

\begin{lem} \label{surj-lem}
The sheaves $q^*q_*(p^*L \otimes I_{\mathcal{Z}})$ and $q^*q_*(p^*L' \otimes I_{\mathcal{Z}})$ are locally free on $X \times \PP$. Further, the natural morphisms $q^*q_*(p^*L \otimes I_{\mathcal{Z}}) \to p^*L \otimes I_{\mathcal{Z}}$ and $q^*q_*(p^*L' \otimes I_{\mathcal{Z}}) \to p^*L' \otimes I_{\mathcal{Z}}$ are surjective.
\end{lem}

\begin{proof}
  We have 
  $$
  i_* q_*\left(p^*L' \otimes I_{\mathcal{Z}}\right) \simeq \hat{q}_* \tau_* \left(p^*L' \otimes I_{\mathcal{Z}}\right)  \simeq  \hat{q}_* \left(\hat{p}^* \hat{L} \otimes \tau_* I_{\mathcal{Z}}\right)\simeq \hat{q}_*\left(\hat{p}^* \hat{L} \otimes  I_{\hat{{\mathcal{Z}}}}\right)
  $$
by  Lemma~\ref{Z-lemma}.
By miracle flatness $I_{\hat{{\mathcal{Z}}}}$ is flat over $\PP_2$.
We have the exact sequence
$$
0 \lra \mathcal{O}_{\hat{X}} \overset{s}\lra  \hat{E} \lra \hat{L} \otimes I_{Z(s)} \lra 0,
$$
and $\mathrm{H}^1(\mathcal{O}_{\hat{X}})=0$, $h^0(\hat{E})=k+3$, so $h^0(\hat{X}, \hat{L} \otimes I_{Z(s)})=k+2$ for any $s\neq 0$, so $q^*q_*(p^*L' \otimes I_{\mathcal{Z}})$ is locally free.
Next,
$$
i_*q_*\left(p^*\left(L'(-\Delta)\right) \otimes I_{\mathcal{Z}}\right) \simeq \hat{q}_* \tau_* \left(p^*\left(L'(-\Delta)\right) \otimes I_{\mathcal{Z}}\right) \simeq  \hat{q}_* \left(\hat{p}^* \hat{L} \otimes I_{\hat{{\mathcal{Z}}}} \otimes \hat{p}^*I_v\right).
$$ 
We have the exact sequence $0 \to I_v \xrightarrow{s}  \hat{E}\otimes I_v \to \hat{L} \otimes I_{Z(s)}\otimes I_v \to 0$. Now $I_v \simeq \mu_* \mathcal{O}_X(-\Delta)$, $ R^j\mu_* \mathcal{O}_X(-\Delta)=0$ for $j>0$ and $\hat{E}\otimes I_v \simeq \mu_*(E(-\Delta))$, so $h^0(I_v)=h^1(I_v)=h^1(\mathcal{O}_X(-\Delta))=0$ and $h^0(\hat{X}, \hat{L} \otimes I_{Z(s)}\otimes I_v)=h^0(X,E(-\Delta))=k+1$, using \cite[Lemma 10]{kemeny-voisin} and the fact that
$q^*q_*(p^*L \otimes I_{\mathcal{Z}})$ is locally free.
Next, we show that
$$
q^*q_*(p^*L' \otimes I_{\mathcal{Z}}) \lra p^*L' \otimes I_{\mathcal{Z}}
$$
is surjective. We have a natural map $\mathrm{H}^0(E) \otimes q^*\mathcal{O}_{\PP}(-1) \to q^*q_*(p^*L' \otimes I_{\mathcal{Z}})$ and a surjection $E \boxtimes \mathcal{O}_{\PP}(-1) \twoheadrightarrow p^*L'\otimes I_{\mathcal{Z}}$. The evaluation map $\mathrm{H}^0(E) \otimes q^*\mathcal{O}_{\PP}(-1) \to E \otimes q^*\mathcal{O}_{\PP}(-1)$ is surjective because $E$ is globally generated. 
We have a commutative diagram
$$\begin{tikzcd}
\mathrm{H}^0(E) \otimes q^*\mathcal{O}_{\PP}(-1) \arrow[r, twoheadrightarrow]  \arrow[d] & E \otimes q^*\mathcal{O}_{\PP}(-1)  \arrow[d, twoheadrightarrow]  \\
q^*q_*(p^*L' \otimes I_{\mathcal{Z}}) \arrow[r]  &p^*L' \otimes I_{\mathcal{Z}}\rlap{,} 
\end{tikzcd}$$
so the lower horizontal map must be surjective. We can show similarly that the map
$$
q^*q_*(p^*L \otimes I_{\mathcal{Z}}) \lra p^*L \otimes I_{\mathcal{Z}}
$$
is surjective. Indeed, we have a natural map
$$
\mathrm{H}^0(E(-\Delta)) \otimes q^*\mathcal{O}_{\PP}(-1) \lra q^*q_*(p^*L \otimes I_{\mathcal{Z}})
$$
and a natural surjection $E(-\Delta)  \boxtimes \mathcal{O}_{\PP}(-1) \hookrightarrow p^*L'(-\Delta) \otimes I_{\mathcal{Z}}=p^*L \otimes I_{\mathcal{Z}}$.
The evaluation map from the group
$\mathrm{H}^0(E(-\Delta)) \otimes q^*\mathcal{O}_{\PP}(-1)$ to $E(-\Delta) \otimes q^*\mathcal{O}_{\PP}(-1)$ is surjective because $E(-\Delta)$ is globally generated; see  \cite[Lemma 10]{kemeny-voisin}. 
\end{proof}
To proceed, we first need a lemma.

\begin{lem} \label{lb}
The sheaves $ R^1 \hat{q}_*(\hat{p}^* \hat{L} \otimes I_{\hat{{\mathcal{Z}}}})$ and $ R^1 \hat{q}_*(\hat{p}^*\hat{L} \otimes I_{\hat{{\mathcal{Z}}}} \otimes  \hat{p}^*I_v)$ on $\PP_2$ are line bundles.
\end{lem}

\begin{proof}
  We have the exact sequence
  $$
  0 \lra \mathcal{O}_{\hat{X}} \overset{t}\lra  \hat{E} \lra \hat{L} \otimes I_{Z(t)} \lra 0
  $$
  for $t \neq 0 \in \mathrm{H}^0(\hat{E})$. Since $\mathrm{H}^1(\hat{E})=\mathrm{H}^2(\hat{E})=0$, this gives $h^1(\hat{X},  \hat{L} \otimes I_{Z(t)})=h^2(\mathcal{O}_{\hat{X}})=1$, so $R^1\hat{q}_*(\hat{p}^*\hat{L} \otimes I_{\hat{{\mathcal{Z}}}})$ is a line bundle by Grauert's theorem. To study $R^1 \hat{q}_*(\hat{p}^*\hat{L} \otimes I_{\hat{{\mathcal{Z}}}} \otimes  \hat{p}^*I_v)$, we use Grauert's theorem plus the exact sequence
  $$
  0 \lra I_v \overset{s}\lra  \hat{E}\otimes I_v \lra \hat{L} \otimes I_{Z(s)}\otimes I_v \lra 0
  $$
of sheaves on $\hat{X}$, for $s \neq 0 \in \mathrm{H}^0(\hat{E})$.  Note that the finite map $\hat{{\mathcal{Z}}} \to \PP_2$ is flat by miracle flatness. Now we claim $\mathrm{H}^i(\hat{X}, \hat{E} \otimes I_v)=0$ for $i=1,2$. Indeed, $\mu_* \mathcal{O}_X(-\Delta) \simeq I_v$ and $R^i\mu_* \mathcal{O}_X(-\Delta) =0$ for $i>0$. 
Thus, $R^i \mu_* (E(-\Delta))\simeq \hat{E} \otimes R^i\mu_* \mathcal{O}_X(-\Delta)=0$ for $i>0$. So $h^i(\hat{X},\hat{E} \otimes I_v)=h^i(X,E(-\Delta))$ for $i>0$. But this vanishes, as $E(-\Delta)$ is a Lazarsfeld--Mukai bundle; see  \cite[Lemma 10]{kemeny-voisin}. So $h^1(\hat{L} \otimes I_{Z(s)}\otimes I_v)=h^2(I_v)=h^2(\mathcal{O}_{\hat{X}})=1$, from the exact sequence
$$0 \lra I_v \lra \mathcal{O}_{\hat{X}} \lra \mathcal{O}_v \lra 0.
$$
So $ R^1 \hat{q}_*(\hat{p}^*\hat{L} \otimes I_{\hat{{\mathcal{Z}}}} \otimes  \hat{p}^*I_v)$ is a line bundle.
\end{proof}

The next result is analogous to \cite[Lemma 1]{kemeny-voisin}.

\begin{prop} \label{W}
Define a coherent sheaf on $\PP$ by $W:=\Coker\left(q_*(p^*L' \otimes I_{\mathcal{Z}}) \to q_*p^*L' \right)$.
Then $W$ is locally free of rank $k+1$. Further, define a coherent sheaf by $\widetilde{W}:=\Coker\left(q_*(p^*L \otimes I_{\mathcal{Z}}) \to q_*p^*L \right)$.
Then $\widetilde{W}$ is also locally free of rank $k+1$.
\end{prop}

\begin{proof}
Define $ W':= \Coker(\hat{q}_*(\hat{p}^*\hat{L} \otimes I_{\hat{{\mathcal{Z}}}} ) \to \hat{q}_*\hat{p}^*\hat{L}  )$ and $\hat{W}:=\Coker(\hat{q}_*(\hat{p}^*\hat{L} \otimes I_{\hat{{\mathcal{Z}}}} \otimes  \hat{p}^*I_v ) \to \hat{q}_*(\hat{p}^*\hat{L}\otimes  \hat{p}^*I_v)  )$
on~$\PP_2$. We will show they are locally free. For the first sheaf,  
 we have the exact sequence
 $$
 0 \lra W' \lra \hat{q}_*\left({\hat{p}^*\hat{L}}_{|_{\hat{{\mathcal{Z}}}}}\right) \lra R^1 \hat{q}_*\left(\hat{p}^*\hat{L} \otimes I_{\hat{{\mathcal{Z}}}}\right) \lra 0.
 $$
 By Lemma \ref{lb}, $R^1 \hat{q}_*(\hat{p}^*\hat{L} \otimes I_{\hat{{\mathcal{Z}}}})$ is a line bundle. The finite map $\hat{{\mathcal{Z}}} \to \PP_2$ is flat by miracle flatness, so, by Grauert's theorem, $ \hat{q}_*({\hat{p}^*\hat{L}}_{|_{\hat{{\mathcal{Z}}}}})$ is a vector bundle of rank $k+2$.  So $W'$ is locally free of rank $k+1$. Define
$S:= \Coker(\hat{q}_*(\hat{p}^*\hat{L} \otimes I_{\hat{{\mathcal{Z}}}} \otimes  \hat{p}^*I_v ) \to \hat{q}_*\hat{p}^*\hat{L})$.
We have an exact sequence
$ 0 \to  \hat{q}_*(\hat{p}^*\hat{L}\otimes  p^*I_v) \to \hat{q}_* \hat{p}^*\hat{L} \to \hat{q}_* (\hat{p}^*\hat{L}_{|_{p \times \PP_2}}) \to 0$,
and $\hat{q}_*(\hat{p}^*\hat{L}_{|_{p \times \PP_2}})$ is a line bundle on $\PP_2$. Indeed, it suffices to show $R^1 \hat{q}_*(\hat{p}^*\hat{L}\otimes  \hat{p}^* I_v)=0$, and, by Grauert's theorem, for this it suffices to note $h^1(X,L)=h^1(X,L'(-\Delta))=h^1(\hat{X},\hat{L}\otimes I_v)=0$; see \cite[Lemma 6]{kemeny-voisin}. So we have an exact sequence
$$
0 \lra \hat{W} \lra S \lra \hat{q}_*\left(\hat{p}^*\hat{L}_{|_{p \times \PP_2}}\right) \lra 0,
$$
where $\hat{q}_*(\hat{p}^*\hat{L}_{|_{p \times \PP_2}})$ is a line bundle. To show $\hat{W}$ is locally free, it suffices to show $S$ is locally free.
 
 Let $T \seq X \times \PP$ denote the closed subscheme $\mathcal{Z} \bigcup \Delta \times \PP $ defined by the product ideal $I_{\mathcal{Z}}  \, p^*I_{\Delta}$. Since $\Delta \times \PP$ is a divisor, this ideal is isomorphic to $I_{\mathcal{Z}} \otimes  p^*I_{\Delta}$. Set $\hat{T} \seq \hat{X} \times \PP_2$ to be the closed  subscheme defined by the product ideal $I_{\hat{{\mathcal{Z}}}} \, \hat{p}^* I_{v}$. Applying the left-exact functor $\tau_*$, and using from Lemma \ref{Z-lemma} the fact that $\tau_*(I_{\mathcal{Z}} \otimes  p^*I_{\Delta})\simeq I_{\hat{{\mathcal{Z}}}} \otimes  \hat{p}^*I_{v}$ and the fact that $\tau_* \mathcal{O}_{X \times \PP}\simeq \mathcal{O}_{\hat{X} \times \PP_2}$ (because, up to isomorphism, we can view $\tau$ as a blow-up in $v \times \PP_2$), we get that the natural map
$I_{\hat{{\mathcal{Z}}}}\otimes  \hat{p}^* I_{v} \to \mathcal{O}_{\hat{X} \times \PP_2}$
is injective, \textit{i.e.} ~$I_{\hat{{\mathcal{Z}}}} \,  \hat{p}^* I_{v} \simeq I_{\hat{{\mathcal{Z}}}}  \otimes \hat{p}^* I_{v}$. So the zero locus of $I_{\hat{{\mathcal{Z}}}} \otimes  I_{v\times \PP_2}$ is $\hat{T}:=\hat{{\mathcal{Z}}} \bigcup \{ v\times \PP_2 \}$.
 We have the exact sequence
 $$
 0 \lra S \lra \hat{q}_*\left({\hat{p}^*\hat{L}}_{|_{\hat{T}}}\right) \lra R^1 \hat{q}_*\left(\hat{p}^*\hat{L} \otimes I_{\hat{{\mathcal{Z}}}} \otimes  I_{p\times \PP_2}\right) \lra 0,
 $$
since $\mathrm{H}^1(\hat{X},\hat{L})=0$. Now $ \hat{q}_*({\hat{p}^*\hat{L}}_{|_{\hat{T}}})$ is a vector bundle of rank $k+3$ by Grauert's theorem. From Lemma \ref{lb}, $ R^1 \hat{q}_*(\hat{L} \otimes I_{\hat{{\mathcal{Z}}}} \otimes  I_{v\times \PP_2})$ is a line bundle. It follows that $S$ is locally free of rank $k+2$ and thus 
$$
\hat{W}:=\Coker\left(\hat{q}_*\left(\hat{p}^*\hat{L} \otimes I_{\hat{{\mathcal{Z}}}} \otimes  \hat{p}^*I_v \right) \lra \hat{q}_*\left(\hat{p}^*\hat{L}\otimes  \hat{p}^*I_v\right)  \right)
$$
is locally free. Now $i \colon \PP\to \PP_2$ is an isomorphism, so the sheaves
$i^* W' \simeq \Coker(i^*\hat{q}_*(\hat{p}^*\hat{L} \otimes I_{\hat{{\mathcal{Z}}}} ) \to i^*\hat{q}_*(\hat{p}^*\hat{L})  )$  and 
$i^* \hat{W}\simeq \Coker(i^*\hat{q}_*(\hat{p}^*\hat{L} \otimes I_{\hat{{\mathcal{Z}}}} \otimes  \hat{p}^*I_v ) \to i^*\hat{q}_*(\hat{p}^*\hat{L}\otimes  \hat{p}^*I_v)  )$
are locally free. But $i$ is an isomorphism, so
$$
i^*\hat{q}_*\left(\hat{p}^*\hat{L} \otimes I_{\hat{{\mathcal{Z}}}} \otimes  \hat{p}^*I_v \right) \simeq q_*\left( p^*L' \otimes I_{Z} \otimes  p^*\mathcal{O}_X(-\Delta) \right) \\
\simeq q_*\left( p^*L\otimes I_{Z}\right).
$$
Likewise $i^*\hat{q}_*(\hat{p}^*\hat{L}\otimes  \hat{p}^*I_v)\simeq q_*( p^*L)$, $i^*\hat{q}_*(\hat{p}^*\hat{L})\simeq q_*( p^*L')$, $i^*\hat{q}_*(\hat{p}^*\hat{L} \otimes I_{\hat{{\mathcal{Z}}}} ) \simeq q_*( p^*L'\otimes I_{Z})$, $i^* \hat{W}\simeq \widetilde{W}$ and $i^* W' \simeq W$, which completes the proof.
\end{proof}

Let $\pi \colon B \to X\times \PP$ be the blow-up along $\mathcal{Z}$, with exceptional divisor $D$, which is smooth since $\mathcal{Z}$ is smooth. Set $p':=p \circ \pi$, $q':= q \circ \pi$. The maps $q^*q_*(p^*L' \otimes I_{\mathcal{Z}}) \to p^*L' \otimes I_{\mathcal{Z}}$ and $q^*q_*(p^*L \otimes I_{\mathcal{Z}}) \to p^*L \otimes I_{\mathcal{Z}}$ are surjective by Lemma \ref{surj-lem}. Notice ${q'}_*({p'}^*L \otimes I_D)\simeq q_*(p^*L \otimes I_{\mathcal{Z}})$, and likewise replacing $L$ by $L'$. Noting $I_D$ is a quotient of $\pi^*I_{\mathcal{Z}}$, we get surjective morphisms ${q'}^*{q'}_*({p'}^*L \otimes I_D) \to {p'}^*L \otimes I_D$ and ${q'}^*{q'}_*({p'}^*L' \otimes I_D) \to {p'}^*L' \otimes I_D$. Since ${q'}^*{q'}_*({p'}^*L \otimes I_D)$ and ${q'}^*{q'}_*({p'}^*L' \otimes I_D)$ are locally free and $D$ is a divisor, we get \emph{vector bundles} $S_1$ and $S_2$ defined by the exact sequences. 
\begin{align*}
0 & \lra S_1 \lra  {q'}^*{q'}_*\left({p'}^*L' \otimes I_D\right) \lra {p'}^*L' \otimes I_D \lra 0\quad\text{and}\\
0 & \lra S_2 \lra {q'}^*{q'}_*\left({p'}^*L \otimes I_D\right) \lra {p'}^*L \otimes I_D \lra 0.
\end{align*}

We have surjections
$$
{q'}^*{q'_*}{p'}^*L' \longtwoheadrightarrow {p'}^*L' \longtwoheadrightarrow {p'}^*L'_{|_D}, \: \: {q'}^*{q'_*}{p'}^*L \longtwoheadrightarrow {p'}^*L\longtwoheadrightarrow {p'}^*L_{|_D}.
$$
These maps induce surjective maps of sheaves
$$
{q'}^*W \longtwoheadrightarrow {p'}^*L'_{|_D}\quad \text{and}\quad {q'}^*\widetilde{W} \longtwoheadrightarrow {p'}^*L_{|_D}.
$$
Define vector bundles $\Gamma_1, \Gamma_1$, both of rank $k+1$, by the exact sequences
$$
0  \lra \Gamma_1 \lra {q'}^*W \lra {p'}^*L'_{|_D} \lra 0 \quad\text{and}\quad 0  \lra \Gamma_2 \lra {q'}^*\widetilde{W} \lra {p'}^*L_{|_D} \lra 0.
$$ 
We have a commutative diagram with exact rows
$$\begin{tikzcd}
0  \arrow[r] & S_2  \arrow[d]  \arrow[r]  & \pi^* \mathcal{M}_L  \arrow[r]  \arrow[d] & \Gamma_2  \arrow[r]  \arrow[d]  & 0 \\
0 \arrow[r] & S_1 \arrow[r] & \pi^* \mathcal{M}_{L'} \arrow[r] & \Gamma_1 \arrow[r] &0\rlap{,}
\end{tikzcd}$$
where the vertical maps are induced from the natural inclusion $L \hookrightarrow L'$ (induced from the effective divisor~$\Delta$) on $X$ and where $\mathcal{M}_L:=p^*M_L$, $\mathcal{M}_{L'}:=p^*M_{L'}$.

We get a natural commutative diagram
$$\begin{tikzcd}
\mathrm{H}^1\left(B, \pi^*\left(\bigwedge^{k+1}\mathcal{M}_{L}(p^*L')\right)\right) \arrow[r, "f"]  \arrow[d]& \mathrm{H}^1\left(B, \bigwedge^{k+1}\Gamma_2\left({p'}^*L'\right)\right) \arrow[d, "g"]\\
\mathrm{H}^1\left(B, \pi^*\left(\bigwedge^{k+1}\mathcal{M}_{L'}\left(p^*L'\right)\right)\right)  \arrow[r, "h"]  & \mathrm{H}^1\left(B, \bigwedge^{k+1}\Gamma_1\left({p'}^*L'\right)\right)\rlap{.}
    \end{tikzcd}$$

\begin{lem} \label{inj-lemma}
The natural map $l \colon \widetilde{W} \to W$, induced by the map $L \hookrightarrow L'$, given by multiplication by a nonzero section $s_{\Delta} \in \mathrm{H}^0(\mathcal{O}_X(\Delta))$, is injective.
\end{lem}

\begin{proof}
Recall $W:=\Coker\left(q_*(p^*L' \otimes I_{\mathcal{Z}}) \to q_*p^*L' \right)$ and $\widetilde{W}:=\Coker\left(q_*(p^*L \otimes I_{\mathcal{Z}}) \to q_*p^*L \right)$.
We have a commutative diagram
$$\begin{tikzcd}
0 \arrow[r]  & \widetilde{W}  \arrow[r] \arrow[d, "l"] & q_* \left(\left(p^*L\right)_{|_{\mathcal{Z}}}\right)   \arrow[r] \arrow[d,  "q_*({\cdot s_{p^*\Delta}}_{|_{\mathcal{Z}}})"] &R^1 q_* \left(p^*L \otimes I_{\mathcal{Z}}\right) \arrow[r] \arrow[d] &0\\
0 \arrow[r]  & W  \arrow[r] &q_* \left(\left(p^*L'\right)_{|_{\mathcal{Z}}}\right) \arrow[r] &R^1 q_* \left(p^*L' \otimes I_{\mathcal{Z}}\right)  \arrow[r] &0\rlap{.}
\end{tikzcd}$$
To prove $l$ is injective, it suffices to show $q_* ({ \cdot { s_{p^*\Delta}}_{|_{\mathcal{Z}} } })$ is injective. Since $q_*$ is left-exact, for this it suffices to show multiplication by ${s_{p^*\Delta}}_{|_{\mathcal{Z}}}$ yields an injective map
$(p^*L)_{|_{\mathcal{Z}}} \to (p^*L')_{|_{\mathcal{Z}}}$, and for this it suffices to note $\mathcal{Z}$ is irreducible and is not contained in $\Delta \times \PP$.
\end{proof}

We now prove the first of three statements needed for the proof. 

\begin{lem} \label{first-step}
The natural morphism $g \colon \mathrm{H}^1(B, \pi^*(\bigwedge^{k+1}\Gamma_2({p'}^*L')) \to  \mathrm{H}^1(B, \bigwedge^{k+1}\Gamma_1({p'}^*L'))$ is an isomorphism. Further, $\det \Gamma_i \simeq {q'}^*\mathcal{O}_{\PP}(k+1)(-D)$ for $i=1,2$.
\end{lem}

\begin{proof}
The natural map $\widetilde{W} \to W$ is injective by Lemma \ref{inj-lemma}. We now prove  $\det \widetilde{W} \simeq \det W$.  Since $q_*p^*L$ and $q_*p^*L'$ are trivial vector bundles, it is equivalent to prove $\det q_*(p^*L \otimes I_{\mathcal{Z}}) \simeq \det q_*(p^*L' \otimes I_{\mathcal{Z}})$, and then  $\det \widetilde{W} \simeq \det W \simeq \det q_*(p^*L \otimes I_{\mathcal{Z}})^{-1}$. We have the exact sequence 
$0 \to \mathcal{O}_X \boxtimes \mathcal{O}_{\PP}(-2) \xrightarrow{\id} E \boxtimes \mathcal{O}_{\PP}(-1) \to p^*L'\otimes I_{\mathcal{Z}} \to 0$, 
and applying $q_*$, we get an exact sequence
$$
0 \lra \mathcal{O}_{\PP}(-2) \lra \mathrm{H}^0(E) \otimes \mathcal{O}_{\PP}(-1) \lra q_*(p^*L'\otimes I_{\mathcal{Z}}) \lra 0.
$$

Since $h^0(E)=k+3$, this gives $\det q_*(p^*L'\otimes I_{\mathcal{Z}}) \simeq \mathcal{O}_{\PP}(-k-1)$. Next, from the exact sequence
$$
0 \lra \mathcal{O}_X(-\Delta) \boxtimes \mathcal{O}_{\PP}(-2) \overset{\id}\lra E(-\Delta)  \boxtimes \mathcal{O}_{\PP}(-1) \lra p^*L'(-\Delta) \otimes I_{\mathcal{Z}} \lra 0
$$
and $h^0(E(-\Delta))=k+1$, see \cite[Lemma 10]{kemeny-voisin}, we get
$$
\det q_*(p^*L \otimes I_{\mathcal{Z}})\simeq \det \mathrm{H}^0(E(-\Delta)) \otimes \mathcal{O}_{\PP}(-1) \simeq \mathcal{O}_{\PP}(-k-1),
$$
as required. From the defining sequences for $\Gamma_i$, $i=1,2$, we see $\det \Gamma_i \simeq \det {q'}^*W (-D)\simeq  {q'}^*\mathcal{O}_{\PP}(k+1)(-D)$. So,
$\bigwedge^{k+1}{\Gamma_2}  \simeq q'^* \det \widetilde{W} (-D) \simeq q'^* \det W(-D) \simeq \bigwedge^{k+1}{\Gamma_1}$. The natural injective map
$$
\bigwedge^{k+1} {\Gamma_2} \longhookrightarrow \bigwedge^{k+1} {q'}^* \widetilde{W} \longhookrightarrow \bigwedge^{k+1} W,
$$
which lands in $\bigwedge^{k+1} {\Gamma_1}$, is an injective map $\bigwedge^{k+1} {\Gamma_2} \hookrightarrow \bigwedge^{k+1} {\Gamma_1}$ between isomorphic line bundles on a projective space and hence must be an isomorphism, as required. 
\end{proof}

We now prove the second of the three needed statements. 

\begin{lem} \label{h}
The map $h\colon \mathrm{H}^1(B, \pi^*(\bigwedge^{k+1}\mathcal{M}_{L'}(p^*L'))  \to \mathrm{H}^1(B, \bigwedge^{k+1}\Gamma_1({p'}^*L'))$ is an isomorphism. 
\end{lem}

\begin{proof}
We follow \cite[Section~2]{kemeny-voisin}. From Lemma \ref{first-step}, $\bigwedge^{k+1}\Gamma_1({p'}^*L') \simeq {q'}^*\mathcal{O}_{\PP}(k+1)) \otimes {p'}^*L' \otimes I_D$.
As in \cite[Corollary 2]{kemeny-voisin},\footnote{Note that $k$ has become $k+1$ as a smooth curve in $|L'|$ has genus $2k+2$} we compute
$$
\mathrm{H}^1\left(B, \bigwedge^{k+1}\Gamma_1\left({p'}^*L'\right)\right) \simeq \mathrm{H}^1\left(X \times \PP, \left(L' \boxtimes \mathcal{O}_{\PP}(k+1)\right) \otimes I_{\mathcal{Z}}\right)
$$
using the sequence
$$
0 \lra \mathcal{O}_X \boxtimes \mathcal{O}_{\PP}(-2) \overset{\id}\lra E \boxtimes \mathcal{O}_{\PP}(-1) \lra p^*L'\otimes I_{\mathcal{Z}} \lra 0.
$$
So $\mathrm{H}^1(B, \bigwedge^{k+1}\Gamma_1({p'}^*L')) \simeq \Sym^{k-1} \mathrm{H}^0(X,E)^{\vee}$.
Now 
$$
h^1\left(\pi^*\left(\bigwedge^{k+1}\mathcal{M}_{L'}(p^*L')\right)\right)=h^1\left(X \times \PP, \bigwedge^{k+1}\mathcal{M}_{L'}(p^*L')\right)= h^1\left(  \bigwedge^{k+1}M_{L'}(L')\right)=\dim  \mathrm{K}_{k,2}(X,L').
$$
Apply verbatim the method of \cite[Section~1]{kemeny-voisin} to prove Green's conjecture for the Gorenstein surface $\hat{X}$:  
$$
\mathrm{K}_{k-1,2}\left(\hat{X}, \hat{L}\right) \simeq \mathrm{H}^1\left(\hat{X}, \bigwedge^k M_{\hat{L}}\left(\hat{L}\right)\right) \simeq  \mathrm{H}^1\left(X, \bigwedge^k M_{L'}(L')\right) \simeq \mathrm{K}_{k-1,2}(X, L')=0.
$$
This vanishing implies $\dim \mathrm{K}_{k,2}(X,L') = \dim  \Sym^{k-1} \mathrm{H}^0(X,E)^{\vee}$; see \cite[Section~4.1]{farkas-progress}. So we have $h^1(\pi^*(\bigwedge^{k+1}\mathcal{M}_{L'}(p^*L'))=h^1( \bigwedge^{k+1}\Gamma_1({p'}^*L'))$. To prove that $h$ is an isomorphism, it suffices to prove it is surjective. From the sequence
$0 \to S_1 \to \pi^* \mathcal{M}_{L'} \to \Gamma_1 \to 0$
of vector bundles on $B$, we obtain the exact sequence
$$
\cdots \lra S_1 \otimes \bigwedge^k \pi^* \mathcal{M}_{L'} \left({p'}^*L'\right) \lra  \bigwedge^{k+1} \pi^* \mathcal{M}_{L'} \left({p'}^*L'\right) \lra \bigwedge^{k+1} \Gamma_1 \left({p'}^*L'\right) \lra 0. 
$$
It suffices to show
$\mathrm{H}^{1+i}(\Sym^i (S_1) \otimes \bigwedge^{k+1-i} \pi^* \mathcal{M}_{L'} ({p'}^*L') )=0$ for $1 \leq i \leq k+1$. As in \cite[Theorem 5]{kemeny-voisin}, we have  the exact sequence
$$
0 \lra S_1 \lra {q'}^*{q'}_*\left({p'}^*L' \otimes I_D\right) \lra {p'}^*L' \otimes I_D \lra 0
$$
and define
$$
\mathcal{G}:={q'}_*\left({p'}^*L' \otimes I_D\right)\simeq q_*\left(p^*L'\otimes I_{\mathcal{Z}}\right)\simeq i^* \hat{q}_*\left(\hat{p}^*\hat{L}\otimes I_{\hat{\mathcal{Z}}}\right) \simeq i^* \hat{\mathcal{G}} 
$$
for  $\hat{\mathcal{G}}:= \hat{q}_*(\hat{p}^*\hat{L}\otimes I_{\hat{\mathcal{Z}}})$. By Grauert's theorem  and the exact sequence $0 \to \mathcal{O}_{\hat{X}} \xrightarrow{s} \hat{E} \to \hat{L} \otimes I_{Z(s)} \to 0$
for any $s \neq 0 \in \mathrm{H}^0(\hat{E})$, we see $\hat{\mathcal{G}}$ and $\mathcal{G}$ are locally free of rank $k+2$. Now the defining sequence for $S_1$ 
gives the exact sequence
$$
0 \lra \Sym^i S_1 \lra   \Sym^i {q'}^* \mathcal{G} \lra  \Sym^{i-1} {q'}^* \mathcal{G}\otimes {p'}^*{L'} \otimes I_D \lra 0
$$
of bundles. It suffices that
\[\mathrm{H}^{1+i}\left(\Sym^i q^* \mathcal{G} \otimes \bigwedge^{k+1-i} \mathcal{M}_{L'} \left(p^*L'\right) \right)=\mathrm{H}^{i}\left(\Sym^{i-1} q^* \mathcal{G} \otimes \bigwedge^{k+1-i} \pi^* \mathcal{M}_{L'} \left(p^*{L'}^2\otimes I_{\mathcal{Z}}\right)\right)=0\]
for $1 \leq i \leq k+1$, or
\[
\mathrm{H}^{1+i}\left(\Sym^i \hat{q}^* \hat{\mathcal{G}} \otimes \bigwedge^{k+1-i} \mathcal{M}_{\hat{L}} \left(\hat{p}^*\hat{L}\right) \right)=\mathrm{H}^{i}\left(\Sym^{i-1} \hat{q}^* \hat{\mathcal{G}} \otimes \bigwedge^{k+1-i} \pi^*  \mathcal{M}_{\hat{L}} \left(\hat{p}^*{\hat{L}}^2\right)\otimes I_{\hat{\mathcal{Z}}}\right)=0\]
for $1 \leq i \leq k+1$. For this, as in \cite[Section~3]{kemeny-voisin}, we can use verbatim the proof of \cite[Theorem~5]{kemeny-voisin}. 
\end{proof}

We can now  prove the main theorem of this paper.

\begin{thm}
Let $X$ denote a K3 surface of Picard rank $2$, generated by a big and nef line bundle $L'$ with $(L')^2=4k+2$, together with the class of a smooth rational curve $\Delta$ with $(L' \cdot \Delta)=0$ for $k\geq 4$, as above. Define the line bundle $L:=L'(-\Delta)$. Then $\mathrm{K}_{k-1,2}(X,L)=0$.
\end{thm}

\begin{proof}
It remains to show $f \colon \mathrm{H}^1(B, \pi^*(\bigwedge^{k+1}\mathcal{M}_{L}(p^*L'))  \to \mathrm{H}^1(B, \bigwedge^{k+1}\Gamma_2({p'}^*L'))$
is surjective. 
From the exact sequence $0 \to S_2 \to \pi^* \mathcal{M}_{L} \to \Gamma_2 \to 0$,
we obtain the exact sequence
$$
\cdots \lra S_2 \otimes \bigwedge^k \pi^* \mathcal{M}_{L} \left({p'}^*L'\right) \lra  \bigwedge^{k+1} \pi^* \mathcal{M}_{L} \left({p'}^*L'\right) \lra \bigwedge^{k+1} \Gamma_2 \left({p'}^*L'\right) \lra 0.
$$
It suffices to show 
$\mathrm{H}^{1+i}(\Sym^i (S_2) \otimes \bigwedge^{k+1-i} \pi^* \mathcal{M}_{L} ({p'}^*L') )=0$ for $1 \leq i \leq k+1$.

Now, we have the exact sequence
$0 \to S_2 \to {q'}^*{q'}_*({p'}^*L \otimes I_D) \to {p'}^*L \otimes I_D \to 0$. Define
$\mathcal{H}:={q'}_*({p'}^*L \otimes I_D)\simeq q_*(p^*L\otimes I_{\mathcal{Z}})$. From the exact sequence
$0 \to \mathcal{O}_X(-\Delta) \boxtimes \mathcal{O}_{\PP}(-2) \xrightarrow{\id} E(-\Delta)  \boxtimes \mathcal{O}_{\PP}(-1) \to p^*L \otimes I_{\mathcal{Z}} \to 0, $
by applying $q_*$, we get $\mathcal{H}\simeq \mathrm{H}^0(X,E(-\Delta)) \otimes \mathcal{O}_{\PP}(-1)$. We have the exact sequence
$ 0 \to \Sym^i S_2 \to   \Sym^i {q'}^* \mathcal{H} \to  \Sym^{i-1} {q'}^* \mathcal{H}\otimes {p'}^*{L'} \otimes I_D \to 0 $.
It suffices to have 
$$
\mathrm{H}^{1+i}\left(\Sym^i {q'}^* \mathcal{H} \otimes \bigwedge^{k+1-i} \pi^* \mathcal{M}_{L} \left(p^*L'\right) \right)=\mathrm{H}^{i}\left(\Sym^{i-1} {q'}^* \mathcal{H} \otimes \bigwedge^{k+1-i} \pi^* \mathcal{M}_{L} \left(p^*{L'}^2\otimes I_{\mathcal{Z}}\right)\right)=0 
$$
for $1 \leq i \leq k+1$. 
So, it suffices to have
$$
\mathrm{H}^{1+i}\left(X \times \PP, \bigwedge^{k+1-i}M_L(L') \boxtimes \mathcal{O}_{\PP}(-i)\right)=\mathrm{H}^{i}\left(X \times \PP, \left(\bigwedge^{k+1-i}M_L\left({L'}^2\right) \boxtimes \mathcal{O}_{\PP}(-i+1)\right)\otimes I_{\mathcal{Z}}\right)=0
$$
for $1 \leq i \leq k+1$, where $\PP \simeq \PP^{k+2}$. The first vanishing is immediate by the K\"unneth formula. For the second, by the above exact sequence, it is enough to have
$$
\mathrm{H}^{i}\left(X \times \PP^{k+2}, \bigwedge^{k+1-i}M_L({L'}\otimes E) \boxtimes \mathcal{O}_{\PP}(-i)\right)=\mathrm{H}^{i+1}\left(X \times \PP^{k+2}, \bigwedge^{k+1-i}M_L({L'} )\boxtimes \mathcal{O}_{\PP}(-i-1)\right)=0 
$$
for $1 \leq i \leq k+1$, which is again immediate by the K\"unneth formula.
\end{proof}

%%%%%%%%%%%%%%%%%%%%%
% References
%%%%%%%%%%%%%%%%%%%%%

\newcommand{\etalchar}[1]{$^{#1}$}

\end{document}